\documentclass[letterpaper,leqno,11pt,oneside]{amsart}

\usepackage{amsmath,amsthm,amsfonts,amssymb}
\usepackage{enumitem}
\usepackage[usenames]{color}
\usepackage{mathtools}
\usepackage[extension=pdf]{hyperref}
\usepackage[height=9.6in,width=5.95in]{geometry}
\usepackage{graphics,graphicx}
\usepackage[final,color,notref,notcite]{showkeys}
\usepackage{xr}

\usepackage[style=alphabetic,natbib=true,maxnames=99,isbn=false,doi=false,url=false,firstinits=true,hyperref=auto,arxiv=abs,backend=bibtex]{biblatex}
\DeclareFieldFormat[article,inbook,incollection,inproceedings,patent,thesis,unpublished]{title}{#1\isdot}
\renewbibmacro{in:}{%
  \ifentrytype{article}{}{%
  \printtext{\bibstring{in}\intitlepunct}}}
\defbibheading{apa}[\refname]{\section*{#1}}

\definecolor{darkblue}{rgb}{0.13,0.13,0.39}
\hypersetup{colorlinks=true,urlcolor=darkblue,citecolor=darkblue,linkcolor=darkblue,pdftitle={Tails of the polymer endpoint distribution}}

\mathtoolsset{showonlyrefs,showmanualtags}

\newtheorem{thm}{Theorem}
\newtheorem{lem}{Lemma}[section]
\newtheorem{prop}[lem]{Proposition}

\theoremstyle{definition}
\newtheorem{rem}[lem]{Remark}
\newtheorem*{rem*}{Remark}

\newcounter{assum}

\newcommand{\pp}{\mathbb{P}}

\newcommand{\ee}{\mathbb{E}}
\newcommand{\rr}{\mathbb{R}}

\newcommand{\zz}{\mathbb{Z}}
\newcommand{\aip}{\mathcal{A}_2}

\newcommand{\ct}{\mathcal{T}}
\newcommand{\cm}{\mathcal{M}}

\newcommand{\p}{\partial}
\newcommand{\uno}[1]{\mathbf{1}_{#1}}

\newcommand{\ep}{\varepsilon}
\newcommand{\vs}{\vspace{6pt}}

\newcommand{\K}{K_{\Ai}}

\newcommand{\qand}{\quad\text{and}\quad}
\newcommand{\qqand}{\qquad\text{and}\qquad}

\DeclareMathOperator{\Ai}{Ai}
\DeclareMathOperator{\tr}{tr}

\newcommand{\gref}[1]{\ref*{g-#1} of \cite{cqr}}
\newcommand{\geqref}[1]{(\ref*{g-#1}) in \cite{cqr}}
\newcommand{\grefn}[1]{\ref*{g-#1}}

\newcommand{\otref}[1]{\ref*{ot-#1} of \cite{qr-airy1to2}}
\newcommand{\oteqref}[1]{(\ref*{ot-#1}) in \cite{qr-airy1to2}}

\DeclareMathOperator*{\argmax}{arg\,max}

\numberwithin{equation}{section}

\let\oldmarginpar\marginpar
\renewcommand\marginpar[1]{\-\oldmarginpar[\raggedleft\footnotesize #1]%
{\raggedright{\small\textsf{#1}}}}

\begin{document}

\title{Tails of the endpoint distribution of directed polymers}

\author{Jeremy Quastel}
\address[J.~Quastel]{
  Department of Mathematics\\
  University of Toronto\\
  40 St. George Street\\
  Toronto, Ontario\\
  Canada M5S 2E4} \email{quastel@math.toronto.edu}
\author{Daniel Remenik}
\address[D.~Remenik]{
  Department of Mathematics\\
  University of Toronto\\
  40 St. George Street\\
  Toronto, Ontario\\
  Canada M5S 2E4 \newline \indent\textup{and}\indent
  Departamento de Ingenier\'ia Matem\'atica\\
  Universidad de Chile\\
  Av. Blanco Encala\-da 2120\\
  Santiago\\
  Chile} \email{dremenik@math.toronto.edu}

\maketitle

\begin{abstract}
  We prove that the random variable $\ct=\argmax_{t\in\rr}\{\aip(t)-t^2\}$, where $\aip$
  is the Airy$_2$ process, has tails which decay like $e^{-ct^3}$. The distribution of
  $\ct$ is a universal distribution which governs the rescaled endpoint of directed
  polymers in $1+1$ dimensions for large time or temperature.
\end{abstract}

\section{Introduction}

Consider the following model of a \emph{directed polymer in a random
  environment}. A \emph{polymer path} is a nearest-neighbor random walk path
$\pi=(\pi_0,\pi_1,\dots)$ in $\zz$ started at the origin, that is,
$\pi_0=0$ and $\pi_k-\pi_{k-1}=\pm1$. On $\zz_+\times\zz$ we place a collection of
independent random weights $\big\{\omega_{i,j}\big\}_{i\geq0,j\in\zz}$. The
\emph{weight} of a polymer path segment $\pi$ of length $N$ is defined as
\[W_N(\pi)=e^{\beta\sum_{k=0}^N\omega_{i,\pi_i}}\] for some fixed $\beta>0$ which is known
  as the \emph{inverse temperature}. If we restrict our attention to paths
of length $N$ which go from the origin to some given $x\in\zz$
then we talk about a \emph{point-to-point polymer}, defined through the path measure
\[Q^{\rm point}_{N,x}(\pi)=\frac{1}{Z^{\rm point}(N,x)}W_{N}(\pi)\] for $\pi$ of length
$N$ going from the origin to $x$ and $Q^{\rm point}_{N,x}(\pi)=0$ otherwise. The
normalizing constant $Z^{\rm point}(N,x)=\sum_{\pi:\,\pi(0)=0,\,\pi(N)=x}W_{N}(\pi)$ is known as the
\emph{point-to-point partition function}. Similarly, if we consider all possible paths of
length $N$ then we talk about a \emph{point-to-line polymer}, defined through the path
measure
\[Q^{\rm line}_N(\pi)=\frac{1}{Z^{\rm line}(N)}W_{N}(\pi)\] for $\pi$ of length $N$ and
$Q^{\rm line}_N(\pi)=0$ otherwise, with the \emph{point-to-line partition function}
$Z^{\rm line}(N)=\sum_{k=-N}^NZ^{\rm point}(N,k)$.

Our main interest will be the point-to-line case. The main quantities of interest in
this case are the partition function and the position of the endpoint of the randomly
chosen path, which we will denote by $\ct_N$. It is widely believed that these
quantities should satisfy the scalings
\begin{equation}
  \log(Z^{\rm line}(N))\sim aN+bN^{1/3}\chi,\label{eq:ZNGOE}\qqand
  \ct_N\sim N^{2/3}\ct,
\end{equation}
where the constants $a$ and $b$ may depend on the distribution of the $\omega_{i,j}$ and
$\beta$, but $\chi$ and $\ct$ should be universal (up to some moment assumptions on the
$\omega_{i,j}$'s).

While there are few results available in the general case described above, the
zero-temperature limit $\beta\to\infty$, known as
\emph{last passage percolation}, is very well understood, at least for some
specific choices of the environment variables $\omega_{i,j}$. We will restrict the
discussion to \emph{geometric last passage percolation}, where one considers a family
$\big\{\omega_{i,j}\}_{i\in\zz^+,j\in\zz}$ of independent geometric random variables with
parameter $q$ (i.e. $\pp(\omega_{i,j}=k)=q(1-q)^{k}$ for $k\geq0$) and
defines the \emph{point-to-point last passage time} by
\[L(N,y)=\max_{\pi:\,\pi(0)=0,\pi(N)=y}\sum_{i=0}^N\omega_{i,\pi(i)}\]
and the \emph{point-to-line last passage time} by
\[L(N)=\max_{y=-N,\dots,N}L(N,y).\]
We remark that this model is usually defined on $(\zz^+)^2$, which corresponds to rotating
our picture by 45 degrees and working on the dual lattice. Although the exact results we
will describe next have been proved for that case, the picture in our situation is morally
the same, and hence for simplicity we present the results for last passage percolation on
$\zz_+\times\zz$.

We define the rescaled process $t\mapsto H_N(t)$ by linearly interpolating the values
given by scaling $L(N,y)$ through the relation
\[L(N,y)=c_1N+c_2N^{1/3}H_N(c_3N^{-2/3}y),\] where the constants $c_i$ have explicit
expressions which depend only on $q$ and can be found in \cite{johansson}. The
point-to-line rescaled process is then given by
\[G(N)=\sup_{t\in[-c_3N^{1/3},c_3N^{1/3}]}H_N(t),\]
and it is known in this case \cite{baikRains} that
\begin{equation}
 G(N)\sim aN+bN^{1/3}\chi\label{eq:LPPGOE}
\end{equation}
with $\chi$ having the Tracy-Widom largest eigenvalue distribution for the Gaussian
Orthogonal Ensemble (GOE) from random matrix theory \cite{tracyWidom2} (the analogous
result holds in the point-to-point case with $\chi$ now having the Tracy-Widom largest
eigenvalue distribution for the Gaussian Unitary Ensemble (GUE) \cite{tracyWidom}). On the
other hand, \citet{johansson} showed that
\begin{equation}
  H_N(t) \to \aip(t)-t^2
\end{equation}
in distribution as $N\to\infty$, in the topology of uniform convergence on compact
sets. Here $\aip$ is the Airy${}_2$ process, which we describe below, and which is a
universal limiting spatial fluctuation process in such models. As a consequence of
Johansson's result (see also \cite{cqr}), \eqref{eq:LPPGOE} translates into
\begin{equation}
  \pp\Big(\sup_{t\in\rr}\big(\aip(t)-t^2\big)\leq m\Big)=F_{\rm GOE}(4^{1/3}m)\label{eq:johGOE}
\end{equation}
(the $4^{1/3}$ arises from scaling considerations, or alternatively from the direct proof
given in \cite{cqr}).

The Airy$_2$ process was introduced by \citet{prahoferSpohn}, and is defined through its
finite-dimensional distributions, which are given by a Fredholm determinant formula: given
$x_0,\dots,x_n\in\mathbb{R}$ and $t_0<\dots<t_n$ in $\mathbb{R}$,
\begin{equation}\label{eq:detform}
\mathbb{P}\!\left(\aip(t_0)\le x_0,\dots,\aip(t_n)\le x_n\right) =
\det(I-\mathrm{f}^{1/2}K_{\mathrm{ext}}\mathrm{f}^{1/2})_{L^2(\{t_0,\dots,t_n\}\times\mathbb{R})},
\end{equation}
where we have counting measure on $\{t_0,\dots,t_n\}$ and
Lebesgue measure on $\mathbb{R}$,  $\mathrm f$ is defined on
$\{t_0,\dots,t_n\}\times\mathbb{R}$ by
$\mathrm{f}(t_j,x)=\uno{x\in(x_j,\infty)}$,
and the {\it extended Airy kernel} \cite{prahoferSpohn,FNH,macedo} is
defined by
\[K_\mathrm{ext}(t,\xi;t',\xi')=
\begin{cases}
\int_0^\infty d\lambda\,e^{-\lambda(t-t')}\Ai(\xi+\lambda)\Ai(\xi'+\lambda), &\text{if $t\ge t'$}\\
-\int_{-\infty}^0 d\lambda\,e^{-\lambda(t-t')}\Ai(\xi+\lambda)\Ai(\xi'+\lambda),  &\text{if $t<t'$},
\end{cases}\] where $\Ai(\cdot)$ is the Airy function. In particular, the one point
distribution of $\aip$ is given by the Tracy-Widom GUE distribution. An alternative
formula for $\aip$ due to \cite{prahoferSpohn}, which is the starting of the proofs given
in \cite{cqr} of \eqref{eq:johGOE} and \eqref{eq:basic} below, and also of the main result
of this paper, is given by
\begin{multline}
  \label{eq:airyfd}
    \pp\!\left(\aip(t_0)\leq x_0,\dotsc,\aip(t_n)\leq x_n\right)\\
  =\det\!\left(I-K_{\Ai}+\bar P_{x_0}e^{(t_0-t_1)H}\bar P_{x_1}e^{(t_1-t_2)H}\dotsm
   \bar P_{x_n}e^{(t_n-t_0)H}K_{\Ai}\right),
\end{multline}
where $K_{\Ai}$ is the \emph{Airy kernel}
\[K_{\Ai}(x,y)=\int_{-\infty}^0 d\lambda\Ai(x-\lambda)\Ai(y-\lambda),\] $H$ is the
\emph{Airy Hamiltonian} $H=-\p_x^2+x$ and $\bar P_a$ denotes the projection onto the
interval $(-\infty,a]$. Here, and in everything that follows, the determinant means the
Fredholm determinant on the Hilbert space $L^2(\rr)$, unless a different Hilbert
space is indicated in the subscript (the last formula \eqref{eq:airyfd} should be compared with
\eqref{eq:detform}, where the Fredholm determinant is computed in an extended space). The
equivalence of \eqref{eq:detform} and \eqref{eq:airyfd} was derived in
\cite{prahoferSpohn} and \cite{prolhacSpohn}. We refer the reader to
\cite{cqr,quastelRemAiry1} for more details.

Coming back to geometric last passage percolation, we turn to the random variables
\[\ct_N=\inf\big\{t\!:\sup_{s\leq t}H_N(s)=\sup_{s\in\rr}H_N(s)\big\},\]
which correspond to the location of the endpoint of the maximizing path with unconstrained
endpoint (that is, the zero-temperature point-to-line polymer). From the above discussion
one expects the following:

\begin{thm}\label{thm:endpointCvgce}
  Let $\ct = \argmax_{t\in \mathbb{R}} \{\aip(t)-t^2\}$. Then, as $N\to\infty$,
  $\ct_N\to\ct$ in distribution.
\end{thm}

This result was proved by \citet{johansson} under the additional hypothesis that the
supremum of $\aip(t)-t^2$ is attained at a unique point. The uniqueness was proved, using
two different methods, by \citet{corwinHammond} and by Moreno Flores and us \cite{mqr}.

Although the result of Theorem \ref{thm:endpointCvgce} has only been proved in the case of
geometric (or exponential) last passage percolation, the key point is that the {\it
  polymer endpoint distribution} is expected to be {\it universal} for directed polymers
in random environments in $1+1$ dimensions, and even more broadly in the KPZ universality
class, for example in particle models such as asymmetric attractive interacting particle
systems (e.g. the asymmetric exclusion process), where second class particles play the
role of polymer paths. This problem has received quite a bit of recent interest in the
physics literature, see \cite{mqr} and references therein for more details.

In \cite{mqr} we obtained an explicit expression for the distribution of $\ct$. More
precisely, we obtained an explicit expression for the joint density of
\[\ct = \argmax_{t\in \mathbb{R}}
\{\aip(t)-t^2\}\qqand\cm=\max_{t\in \mathbb{R}} \{\aip(t)-t^2\},\] which we will denote as
$f(t,m)$.  To state the formula we need some definitions. Let $B_m$ be the integral
operator with kernel
\begin{equation}
  \label{eq:Bc}
  B_m(x,y)=\Ai(x+y+m).
\end{equation}
Recall that \citet{ferrariSpohn} showed that $F_\mathrm{GOE}$ can be expressed as the
determinant
\begin{equation}
  F_\mathrm{GOE}(m)=\det(I-P_0B_mP_0),\label{eq:GOE}
\end{equation}
  where $P_a$ denotes the projection onto the interval $[a,\infty)$ (the formula
  essentially goes back to \cite{sasamoto}). In particular, note that since
  $F_\mathrm{GOE}(m)>0$ for all $m\in\rr$, \eqref{eq:GOE} implies that $I-P_0B_mP_0$ is
  invertible. For $t,m\in\rr$ define the function
\begin{equation}
  \label{eq:phi}
  \psi_{t,m}(x)=2e^{xt}\left[t\Ai(x+m+t^2)+\Ai'(x+m+t^2)\right]
\end{equation}
and the kernel
\[\Psi_{t,m}(x,y)=2^{1/3}\psi_{t,m}(2^{1/3}x)\psi_{-t,m}(2^{1/3}y).\]
Then the joint density of $\ct$ and $\cm$ is given by
\begin{equation}
  \label{eq:ftm}
  \begin{split}
      f(t,m)&=\det\!\big(I-P_0B_{4^{1/3}m}P_0+P_0\Psi_{t,m}P_0\big)-F_\mathrm{GOE}(4^{1/3}m)\\
      &=\tr\!\big[(I-P_0B_{4^{1/3}m}P_0\big)^{-1}P_0\Psi_{t,m}P_0\big]F_\mathrm{GOE}(4^{1/3}m).      
    \end{split}
\end{equation}
Integrating over $m$ one obtains a formula for the probability density $f_{\rm
  end}(t)$ of $\ct$, although it does not appear that the resulting integral can be
computed explicitly. One can readily check nevertheless that $f_{\rm end}(t)$ is symmetric in $t$.
Figure \ref{fig:density}, taken from \cite{mqr}, shows a plot of the marginal $\ct$ density.

The goal of this paper is to study the decay of the tails of $\ct$. We will prove:

\begin{thm}\label{thm:tail}
 There is a $c>0$ such that for every $\kappa>\frac{32}3$ and large enough $t$,
 \[e^{-\kappa t^3}\leq\pp\big(|\ct|>t\big)\leq
 ce^{-\frac43t^3+2t^2+\mathcal{O}(t^{3/2})}.\]
\end{thm}

We believe that the correct exponent is the $-\frac43$ obtained in the upper bound (we
remark that we have not attempted to get better estimates on the lower order terms in the
upper bound). The tail decay of order $e^{-ct^3}$ confirms a prediction made in the
physics literature in \citet{halpZhang}, see also \citet{mezardParisi}. Their idea is to
argue by analogy with the argmax of Brownian motion minus a parabola. In that case one has
a complete analytical solution \cite{groeneboom}.

We will give two proofs of the upper bound, both in Section \ref{sec:upper}. The first one
is based on a direct application of the formula \eqref{eq:ftm} for the joint
density of $\ct$ and $\cm$. The second proof will start from a probabilistic argument and
then use the continuum statistics formula for the Airy$_2$ process obtained in \cite{cqr}
to estimate the probability that the maximum is attained very far from the origin. This
formula corresponds to the continuum limit of \eqref{eq:airyfd} and is given as follows
(see \cite{cqr} for more details). Fix $\ell<r$ and $g\in H^1([\ell,r])$ and define an
operator $\Theta^g_{[\ell,r]}$ acting on $L^2(\rr)$ by
$\Theta^g_{[\ell,r]}f(\cdot)=u(r,\cdot)$, where $u(r,\cdot)$ is the solution at time $r$
of the boundary value problem
\begin{equation}
  \begin{aligned}
    \p_tu+Hu&=0\quad\text{for }x<g(t), \,\,t\in (\ell,r)\\
    u(\ell,x)&=f(x)\uno{x<g(\ell)}\\
    u(t,x)&=0\quad\text{for }x\ge g(t).
  \end{aligned}\label{eq:bdval}
\end{equation}
Then
\begin{equation}\pp\!\left(\aip(t)\leq g(t)\text{ for
      }t\in[\ell,r]\right)=\det\!\left(I-K_{\Ai}+\Theta^g_{[\ell,r]}  e^{(r-\ell)H}K_{\Ai}\right).\label{eq:basic}
\end{equation}
We remark that in the second proof actually get an upper bound with a larger $\mathcal{O}(t^2)$
correction in the exponent.

Not surprisingly, the lower bound turns out to be more difficult (in fact, for the upper
bound we can basically use the estimate
$|\hspace{-0.05em}\det(I+A)-\det(I+B)|\leq\|A-B\|_1e^{1+\|A\|_1+\|B\|_1}$ for trace class operators $A$
and $B$ directly to estimate the decay by computing the trace norm of two operators; no
such estimate is available for the lower bound). In this case we will have to use a
probabilistic argument to extract the lower bound from the well-known exact asymptotics
for the tails of the GUE distribution, and then show that the remaining terms are of lower
order. For this last task we will use again \eqref{eq:basic}, but the argument is much
more complicated than for the upper bound. Interestingly, it will involve turning an
instance of \eqref{eq:basic} which mixes continuum and discrete statistics for $\aip$ back
into an extended kernel formula.

\begin{rem}\mbox{}
  \begin{enumerate}[label=\arabic*.]
  \item A few days before submitting this article, we became aware of the very recent work
    of \citet{schehr}, where he obtains, using non-rigorous arguments, an alternative
    formula for the joint distribution function of $\cm$ and $\ct$. His formula is
    obtained by taking the limit in $N$ of a known formula for the joint distribution of
    the maximum and location of the maximum for the top line of $N$ non-intersecting
    Brownian excursions, which is expected to converge to the Airy$_2$ process. The
    resulting formula is expressed in terms of quantities associated to the
    Hastings-McLeod solution of the Painlev\'e II equation, and has tails decaying like
    $e^{-\frac43t^3}$.
  \item During the refereeing process, \citet{baikLiechtySchehr} proved the equivalence of
    the formula of \cite{schehr} and \eqref{eq:ftm}.  Hence the rigorous validity of the
    formula of \cite{schehr} is established based on \cite{mqr}, as well as the tail
    decay.
  \end{enumerate}
\end{rem}

\vs
\paragraph{\bf Acknowledgements}  The authors would like to thank the referee for a careful
reading of the article.
Both authors were supported by the Natural Science and Engineering Research Council of
Canada, and DR was supported by a Fields-Ontario Postdoctoral Fellowship.

\section{Upper bound}
\label{sec:upper}

Throughout the paper $c$ and $C$ will denote positive constants whose values may change
from line to line. We will denote by $\|\cdot\|_{\rm op}$, $\|\cdot\|_1$ and $\|\cdot\|_2$
respectively the operator, trace class and Hilbert-Schmidt norms of operators on
$L^2(\rr)$ (see Section \gref{sec:aiL} for the definitions or \cite{simon} for a complete
treatment). We will use the following facts repeatedly (they can all be found in
\cite{simon}): if $A$ and $B$ are bounded linear operators on $L^2(\rr)$, then
\begin{equation}\label{eq:norms}
  \begin{gathered}
    \|AB\|_1\leq\|A\|_2\|B\|_2,\qquad\|AB\|_2\leq\|A\|_{\rm
      op}\|B\|_2,\qquad\|AB\|_2\leq\|A\|_2\|B\|_{\rm op},\\
    \|A\|_{\rm op}\leq\|A\|_2\leq\|A\|_1,\\
    \|A\|_2^2=\int_{\rr^2}dx\,dy\,A(x,y)^2,
  \end{gathered}
\end{equation}
where in the last one we are assuming that $A$ has integral kernel $A(x,y)$. We will also
use the bound
 \begin{equation}
    \label{eq:tracecont}
    \left|\det(I+A)-\det(I+B)\right|\leq\|A-B\|_1e^{\|A\|_1+\|B\|_1+1}\leq\|A-B\|_1e^{\|A-B\|_1+2\|B\|_1+1}
\end{equation}
for any two trace class operators $A$ and $B$.

We recall that the shifted Airy functions $\phi_\lambda(x)=\Ai(x-\lambda)$ are the
generalized eigenfunctions of the Airy Hamiltonian, as
$H\phi_\lambda=\lambda\phi_\lambda$, and the Airy kernel $K_{\Ai}$ is the projection of
$H$ onto its negative generalized eigenspace (see Remark \gref{airyrem}). This implies
that $e^{sH}\K$ has integral kernel
\begin{equation}
  \label{eq:eLHK}
  e^{sH}\K(x,y)=\int_0^{\infty}d\lambda\,e^{-s\lambda}\Ai(x+\lambda)\Ai(y+\lambda).
\end{equation}
It also implies that $e^{aH}\K e^{bH}\K=e^{(a+b)H}\K$. We will use this fact several times
in this and the next section.

\subsection{First proof}

We start by writing
\[\pp(\ct>t)\leq\pp(\ct>t,\cm>-2t)+\pp(\cm\leq-2t).\]
By \eqref{eq:johGOE} the second probability on the right side equals $F_{\rm
  GOE}(-2^{5/3}t)\leq ce^{-\frac43t^3}$, where the tail bound can be found in
\cite{baikBuckDiF}. Thus it will be enough to prove that
\begin{equation}
  \label{eq:ctcmbd}
  \pp(\ct>t,\cm>-2t)\leq ce^{-\frac43t^3+2t^2+\mathcal{O}(t^{3/2})}.
\end{equation}

Let $s\geq t$. We will assume for the rest of the proof that $m>-2t$.  Using
\eqref{eq:GOE} and the first formula in \eqref{eq:ftm} we get from \eqref{eq:tracecont}
that
\begin{equation}
f(s,m)\leq\|P_0\Psi_{s,m}P_0\|_1e^{1+2\|P_0B_{4^{1/3}m}P_0\|_1+\|P_0\Psi_{s,m}P_0\|_1}.\label{eq:ftmbd}
\end{equation}
Using the identity
\begin{equation}
  \label{eq:airyConv}
  \int_{-\infty}^\infty du\Ai(a+u)\!\Ai(b-u)=2^{-1/3}\Ai(2^{-1/3}(a+b))
\end{equation}
and letting $\ep=t^{-1}$ we may write
\begin{equation}
  \begin{aligned}
    P_0B_{4^{1/3}m}P_0=2^{1/3}Q_1Q_2\qquad\text{with}\quad
    Q_1(x,\lambda)&=\uno{x\geq0}\Ai(2^{1/3}x+m+\lambda)e^{\frac12\ep\lambda},\\
    Q_2(\lambda,y)&=e^{-\frac12\ep\lambda}\Ai(2^{1/3}y+m-\lambda)\uno{y\geq0}.
  \end{aligned}\label{eq:A11A12}
\end{equation}
Lemma \ref{lem:Qs} now gives
\begin{equation}
  \|P_0B_{4^{1/3}m}P_0\|_1\leq ct^{3/2}.\label{eq:Bmbd}
\end{equation}

On the other hand, recall that the trace norm of an operator $\Psi$ acting on $L^2(\rr)$
is defined as
\[\|\Psi\|_1=\sum_{n=1}^{\infty}\langle e_n,|\Psi|e_n\rangle,\] where $\{e_n\}_{n\geq 1}$
is any orthonormal basis of $L^2(\rr)$ and $|\Psi|=\sqrt{\Psi^*\Psi}$ is the unique
positive square root of the operator $\Psi^*\Psi$. For the case $\Psi=P_0\Psi_{s,m}P_0$,
since $\Psi$ is a rank one operator it is easy to check that $\Psi^*\Psi$ has only one
eigenvector, and in fact it is given by $\uno{x\geq0}\psi_{-s,m}(2^{1/3} x)$ with
associated eigenvalue
$\lambda_{s,m}=2^{1/3}\|P_0\psi_{s,m}(2^{1/3}\cdot)\|^2_2\|P_0\psi_{-s,m}(2^{1/3}\cdot)\|^2_2$. We
deduce that $\|P_0\Psi_{s,m}P_0\|_1=\sqrt{\lambda_{s,m}}$, and then by \eqref{eq:ftmbd}
and \eqref{eq:Bmbd} we get
\begin{equation}
  \int_{-2t}^\infty\,dm f(s,m)\leq \int_{-2t}^\infty dm\,\sqrt{\lambda_{s,m}}e^{1+ct^{3/2}+\sqrt{\lambda_{s,m}}}.\label{eq:intminus2t}
\end{equation}
Now Lemma \ref{lem:psi2bd} gives
\begin{align}
  \int_{-2t}^\infty dm\,\sqrt{\lambda_{s,m}}&=2^{1/6}\int_{-2t}^\infty
  dm\,\|P_0\psi_{s,m}(2^{1/3}\cdot)\|_2\|P_0\psi_{-s,m}(2^{1/3}\cdot)\|_2\\
  &\leq2^{1/6}\left[\int_{-2t}^\infty
    dm\,\|P_0\psi_{s,m}(2^{1/3}\cdot)\|^2_2\right]^{1/2}\left[\int_{-2t}^\infty
    dm\,\|P_0\psi_{-s,m}(2^{1/3}\cdot)\|^2_2\right]^{1/2}\\
  &\leq ce^{-\frac43s^3+2st,}
\end{align}
and it is not hard to see from the proof of Lemma \ref{lem:psi2bd} that $\lambda_{s,m}$ is
bounded uniformly for $m\geq-2t$, $s>t$ and large enough $t$. We deduce then from
\eqref{eq:intminus2t} that
$\int_{-2t}^\infty dm\,f(s,m)\leq ce^{-\frac43s^3+2st+\mathcal{O}(t^{3/2})}$, and hence
\begin{align}
  \pp(\ct>t,\,\cm>-2t)&=\int_t^\infty ds\int_{-2t}^\infty dm\,f(s,m)\leq c\int_t^\infty
  ds\,e^{-\frac43s^3+2st+\mathcal{O}(t^2)}\\&\leq ce^{-\frac43t^3+2t^2+\mathcal{O}(t^{3/2})},
\end{align}
where the last estimate can be easily obtained from an application of Laplace's method,
see the proof of Lemma \ref{lem:airyInt} for a similar estimate. This gives
\eqref{eq:ctcmbd} and the upper bound of Theorem \ref{thm:tail}.

\begin{figure}
  \centering \hspace{-0.0in}\includegraphics[width=4in]{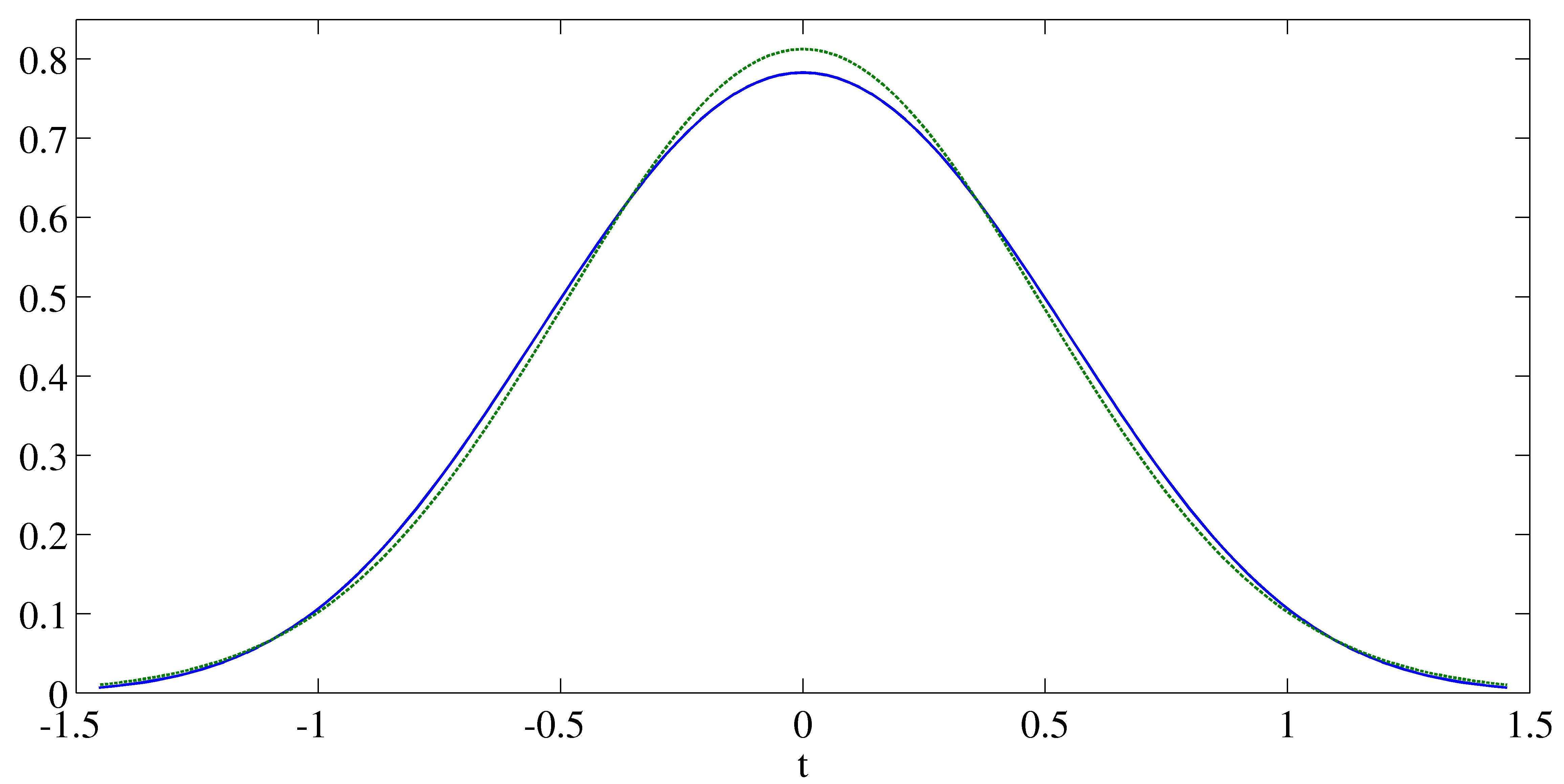} \vspace{-14pt}
  \caption{Plot of the density of $\ct$ compared with a Gaussian density with the same
    variance 0.2409 (dashed line). The excess kurtosis $\ee(\ct^4)/\ee(\ct^2)^2-3$ is
    $-0.2374$.}
  \label{fig:density}
\end{figure}

\subsection{Second proof}

Since we already have a full proof of the upper bound, we will skip some details.
The key result for this proof is the following

\begin{prop}\label{prop:supfin}
  Fix $L\geq1$. Then there is a $c>0$ such that for every $m>0$
  \[\pp\!\left(\sup_{x\in[-L,L]}\aip(x)>m+1\right)\leq c\,e^{-\frac43m^{3/2}}.\]
\end{prop}

\begin{proof}
  By \eqref{eq:basic} we have, writing $g_m(s)=s^2+m$,
  \begin{equation}
    \begin{split}
      \pp\!\left(\sup_{x\in[-L,L]}\aip(x)\leq m+1\right)
      &\geq\pp\!\left(\sup_{x\in[-L,L]}\big(\aip(x)-x^2\big)\leq m\right)\\
      &=\det\!\left(I-K_{\Ai}+e^{LH}K_{\Ai}\Theta^{g_{m}}_{[-L,L]}e^{LH}K_{\Ai}\right),
    \end{split}\label{eq:supfin1}
  \end{equation}
  where we have used the cyclic property of the determinant together with the identity
  $e^{2LH}\K=e^{LH}\K e^{LH}\K$ (see the remark after \eqref{eq:eLHK}). Now recall from
  Theorem \gref{thm:goe} that
  \[F_{\rm GOE}(4^{1/3}m)=\det\!\big(I-e^{LH}K_{\Ai}R_Le^{LH}K_{\Ai}\big),\] where $R_L$ is
  defined in \geqref{eq:RL}.  Therefore by \eqref{eq:tracecont} we deduce that
  \begin{equation}\label{eq:GOEbd}
    \det\!\left(I-K_{\Ai}+e^{LH}K_{\Ai}\Theta^{g_m}_{[-L,L]}e^{LH}K_{\Ai}\right)
    \geq F_{\rm GOE}(4^{1/3}m)-\|A-B\|_1e^{1+\|A-B\|_1+2\|B\|_1},
  \end{equation}
  where
  \[A=K_{\Ai}-e^{LH}K_{\Ai}\Theta^{g_m}_{[-L,L]}e^{LH}K_{\Ai}\qqand
  B=e^{LH}K_{\Ai}R_Le^{LH}K_{\Ai}\] and we have used the triangle inequality in the
  exponent.  Now $\|B\|_1$ can easily be bounded by some constant uniformly in $m>0$ by an
  argument similar to the one used to obtain \eqref{eq:Bmbd}. On the other hand, using the
  decomposition of $\Theta^{g_m}_{[-L,L]}$ given in \geqref{eq:diff1} we have
  \[A-B=e^{LH}K_{\Ai}\Omega_Le^{LH}K_{\Ai},\] where $\Omega_L=\big(R_L-\bar
  P_{m+L^2}R_L\bar P_{m+L^2}\big) -\big(e^{-2LH}-\bar P_{m+L^2}e^{-2LH}\bar
  P_{m+L^2}\big)$. By Lemma \ref{lem:tracenormbd} we get $\|A-B\|_1\leq c\,e^{-\eta
    m^{3/2}}$ for some $\eta>\frac43$, and then using this and \eqref{eq:GOEbd} in
  \eqref{eq:supfin1} we obtain
  \[\pp\!\left(\sup_{x\in[-1,1]}\aip(x)\leq m+1\right)
  \geq F_{\rm GOE}(4^{1/3}m)-c\,e^{-\eta m^{3/2}}.\] The result follows from this and
  the asymptotics \cite{baikBuckDiF} $F_{\rm GOE}(4^{1/3}m)\geq1-cm^{-3/2}e^{-\frac43m^{3/2}}$.
\end{proof}

Using Proposition \ref{prop:supfin} we can derive the upper bound. Start by observing that, for fixed
$\sigma\in(0,1)$,
  \begin{equation}
    \begin{split}
      \pp\!\left(\ct\in[s,s+2]\right)&\leq\pp\!\left(\sup_{x\in[s,s+2]}\big(\aip(x)-x^2\big)>\aip(0)\right)
      \leq\pp\!\left(\sup_{x\in[s,s+2]}\aip(x)>\aip(0)+s^2\right)\\
      &\leq\pp\!\left(\sup_{x\in[s,s+2]}\aip(x)>(1-\sigma)s^2\right)
      +\pp\!\left(\aip(0)<-\sigma s^2\right).
    \end{split}
    \label{eq:suptt}
  \end{equation}
  The last probability equals $F_{\rm GUE}(-\sigma s^2)$, and then the asymptotics obtained
  in \cite{baikBuckDiF} give $\pp\!\left(\aip(0)<-\sigma s^2\right)\leq
  c\,e^{-\frac1{12}\sigma^{3}s^6}$.  For the other term on the right side of
  \eqref{eq:suptt} we use Proposition \ref{prop:supfin} and the stationarity of the
  Airy$_2$ process to write, for $s\geq t$,
  \[\pp\!\left(\sup_{x\in[s,s+2]}\aip(x)>(1-\sigma)s^2\right)\leq
  c\,e^{-\frac43[(1-\sigma)s^2-1]^{3/2}}.\] Therefore
  \[\pp\!\left(\ct\in[s,s+2]\right)\leq
  c\,e^{-\min\{\frac1{12}\sigma^{3}s^6,\frac43[(1-\sigma)s^2-1]^{3/2}\}}.\] 
  Take $\sigma=4^{2/3}s^{-1}$ so that the minimum in the above exponent equals
  $\frac43[s^2-4^{2/3}s-1]^{3/2}=\frac43s^3+\mathcal{O}(s^2)$ for large enough $s$. We deduce that
  \[\pp\!\left(\ct\in[s,s+2]\right)\leq c\,e^{-\frac43s^3+\mathcal{O}(s^{2})}.\]
  Summing this inequality over intervals of the form $[t+2k,t+2(k+1)]$ for $k\geq0$ gives
  \[\pp\big(\ct>t\big)\leq c\,e^{-\frac43t^3+\mathcal{O}(t^{2})}\]
  for some $c>0$ and large enough $t$, and now the upper bound in Theorem \ref{thm:tail}
  (with a worse $\mathcal{O}(t^2)$ correction) follows from the symmetry of $\ct$.

\section{Lower bound}
\label{sec:lowerbd}

As we mentioned, the lower bound turns out to be more
delicate, because in this case a simple bound like \eqref{eq:tracecont} is not
available. The main idea of the proof is to compare the probability we are interested in
with an expression involving the one-dimensional marginal of the Airy$_2$ process, and then
extract the lower bound from known asymptotics for the Tracy-Widom GUE distribution. This
comparison will introduce an error term, and most of the work in the proof will be to show
that this error term is of lower order.

The first step in the comparison is to write, for $t>0$, $\beta\geq0$ and $s>0$,
\begin{equation}
  \pp(|\ct|>t)\geq\pp\!\left(\aip(x)-x^2\leq\beta t^2\,\,\forall\,x\leq
    t,\,\aip(t+s)-(t+s)^2>\beta t^2\right).\label{eq:firstbdlowpre}
\end{equation}
The idea is the following.  If $s$ is now taken to be reasonably large, then $\aip(t+s)$
and $\aip(x)$, $x\leq t$, should have decorrelated somewhat. Assuming they have completely
decorrelated, we would have that the right side is bounded below by
$\pp\!\left(\aip(t+s)-(t+s)^2>\beta t^2\right)$, which has the correct decay if we choose
$s=\alpha t$. So the whole proof comes down to estimating the correction coming from the
correlation.

Of course, from \eqref{eq:firstbdlowpre} we have
\begin{multline}
  \pp(|\ct|>t)
  =\pp\!\left(\aip(x)-x^2\leq\beta t^2\,\,\forall\,x\leq t\right)\\
  -\pp\!\left(\aip(x)-x^2\leq\beta t^2\,\,\forall\,x\leq t,\,\aip(t+s)-(t+s)^2\leq\beta t^2\right),\label{eq:firstbdlow}
\end{multline}
The bound will then follow from the following

\begin{lem}\label{lem:prtbd}
  Let $\beta>3$. There is an $\alpha_0>0$ (which depends on $\beta$) such that if
  $\alpha\in(0,\alpha_0)$ and $s=\alpha t$, then for large enough $t$ we have
  \begin{multline}
    \pp\!\left(\aip(x)-x^2\leq\beta t^2\,\,\forall\,x\leq t,\,\aip(t+s)-(t+s)^2\leq\beta t^2\right)\\
    \leq\pp\!\left(\aip(x)-x^2\leq\beta t^2\,\,\forall\,x\leq
      t\right)\pp\!\left(\aip(t+s)-(t+s)^2\leq\beta t^2\right)\\
    \cdot\big[1+\tfrac12a_2(t+s)^{-3}e^{-\frac43(\beta+1)^{3/2}(t+s)^3}\big],
  \end{multline}
  where $a_2$ is defined implicitly in \eqref{eq:bdGUE}.
\end{lem}

To see how the lower bound follows from this, let $\beta>3$ and choose $\alpha$ as in the
lemma. Then letting $s=\alpha t$ and using the lemma and \eqref{eq:firstbdlow} we get
\begin{multline}
  \pp(|\ct|>t)\geq\pp\!\left(\aip(x)-x^2\leq\beta t^2\,\,\forall\,x\leq t\right)\\\cdot
  \left(1-\pp\!\left(\aip(t+s)-(t+s)^2\leq\beta t^2\right)\big[1
    +\tfrac1{2}a_2(t+s)^{-3}e^{-\frac43(\beta+1)^{3/2}(t+s)^3}\big]\right).
\end{multline}
Now if we let $p_0=\pp\!\left(\aip(x)-x^2\leq0\,\,\forall\,x\in\rr\right)=F_{\rm
  GOE}(0)>0$ then, since $t>0$, the first probability on the right side above is larger than $p_0$. On the other
hand,
\begin{equation}
  \begin{aligned}
  \pp\!\left(\aip(t+s)-(t+s)^2\leq\beta t^2\right)&=F_{\rm
    GUE}((\beta+1)t^2+2st+s^2)\leq F_{\rm
    GUE}((\beta+1)(t+s)^2)\\&
  \leq1-a_2(t+s)^{-3}e^{-\frac43(\beta+1)^{3/2}(t+s)^3}
\end{aligned}
\label{eq:bdGUE}
\end{equation}
for some explicit constant $a_2>0$ and large enough $t$, see for instance
\cite{baikBuckDiF} for the precise bounds on the tails of the GUE distribution. This
implies that
\begin{align}
  \pp(|\ct|>t)&\geq
  p_0\!\left(\tfrac1{2}a_2(t+s)^{-3}e^{-\frac43(\beta+1)^{3/2}(t+s)^3}+\tfrac12a_2^2(t+s)^{-6}e^{-\frac83(\beta+1)^{3/2}(t+s)^3}\right)\\
  &\geq c(1+\alpha)^{-3}t^{-3}e^{-\frac43(1+\beta)^{3/2}(1+\alpha)^3t^3},
\end{align}
and now the lower bound in Theorem \ref{thm:tail} follows from choosing $\beta>3$ and a
small enough $\alpha>0$ so that $\frac43(1+\beta)^{3/2}(1+\alpha)^3<\kappa$.

Our goal then is to prove Lemma \ref{lem:prtbd}. For this we need an expression for the
probability we want to bound. The answer follows from a simple extension of the result in
\cite{qr-airy1to2}, where we obtained an explicit expression for probabilities of the form
\[\pp\!\left(\sup_{x\leq t}(\aip(x)-x^2)\leq m\right)\]
and showed that they to correspond (after a suitable shift) to the one-dimensional
marginals of the Airy$_{2\to1}$ process \cite{bfs}. To state the extension of that
formula, 
define for $a,t\in\rr$ the operators
\[\varrho_{a,t}f(x)=f(2(a+t^2)-x)\qqand M_{a,t}f(x)=e^{2t(x-a- t^2)}f(x)\]
acting on $f\in L^2(\rr)$. We will say that an operator acting on $L^2(\rr)$ is
\emph{identity plus trace class} if it can be written in the form $I+A$ with $A$ a trace
class operator.

\begin{lem}\label{lem:1to2}
  With the above definitions, and for any $a,b\in\rr$ and $s>0$,
  \begin{multline}\label{eq:det1to2}
    \pp\!\left(\aip(x)-x^2\leq a\,\,\forall\,x\leq t,\,\aip(t+s)-(t+s)^2\leq b\right)\\
    =\det\!\left(I-\K+\K(I-M_{a,t}\varrho_{a,t})\bar P_{a+t^2}e^{-sH}\bar P_{b+(t+s)^2}e^{sH}\K\right).
  \end{multline}
  Moreover, the operator inside this determinant is identity plus trace class.
\end{lem}

\begin{proof}
  For $L>0$ it is straighforward to adapt the proof given in \cite{cqr} of the continuum
  statistics formula \eqref{eq:basic} to deduce that
  \begin{multline}
    \pp\!\left(\aip(x)-x^2\leq a\,\,\forall\,x\in[-L,t],\,\aip(t+s)-(t+s)^2\leq b\right)\\
    =\det\!\left(I-\K+\Theta_{[-L,t]}e^{-sH}\bar P_{b+(t+s)^2}e^{(L+t+s)H}\K\right),
  \end{multline}
  where $\Theta_{[-L,t]}$ is defined as $\Theta^{g}_{[-L,t]}$ (see \eqref{eq:bdval}) for
  $g(x)=x^2+a$.  Since $e^{(L+t+s)H}\K=e^{sH}\K e^{(L+t)H}\K$ and $\K=e^{sH}\K
  e^{-sH}\K=e^{-sH}\K e^{sH}\K$ (see the remark after \eqref{eq:eLHK}), we can use the
  cyclic property of the determinant to turn the last determinant into
  \begin{equation}
    \det\!\left(I-\K+e^{(L+t)H}\K\Theta_{[-L,t]}e^{-sH}\bar P_{b+(t+s)^2}e^{sH}\K\right).\label{eq:det2}
  \end{equation}
  Rewriting the above operator as
  \begin{equation}
    I-e^{-sH}\K
    P_{b+(t+s)^2}e^{sH}\K-e^{(L+t)H}\K\big(e^{-(L+t)H}-\Theta_{[-L,t]}\big)e^{-sH}\bar
    P_{b+(t+s)^2}e^{sH}\K\label{eq:reop},
  \end{equation}
  it follows from an easy adaptation of the proof of Proposition
  \gref{prop:theta} that the operator is identity plus trace class. On the other hand, in the proof of
  Theorem \otref{thm:1to2} it was shown that
  \[e^{(L+t)H}\K\Theta_{[-L,t]}\xrightarrow[L\to\infty]{}\K(I-M_{a,t}\varrho_{a,t})\bar
  P_{a+t^2}\] in Hilbert-Schmidt norm, and a straightforward extension of the proof
  shows that the same holds if we post-multiply both sides by $e^{-sH}\bar
  P_{b+(t+s)^2}N$, where $Nf(x)=(1+x^2)^{1/2}f(x)$. Thus, since $N^{-1}e^{sH}\K$ is
  Hilbert-Schmidt by \geqref{eq:sndHS} deduce by \eqref{eq:norms} that
  \[e^{(L+t)H}\K\Theta_{[-L,t]}e^{-sH}\bar P_{b+(t+s)^2}e^{sH}\K
  \xrightarrow[L\to\infty]{}\K(I-M_{a,t}\varrho_{a,t})\bar P_{a+t^2}e^{-sH}\bar
  P_{b+(t+s)^2}e^{sH}\K\] in trace class norm. This together with \eqref{eq:det2} and
  \eqref{eq:reop} yields \eqref{eq:det1to2} and, in particular, the fact that the operator inside
  this determinant is identity plus trace class.
\end{proof}

The key to obtain Lemma \ref{lem:prtbd} from Lemma \ref{lem:1to2} is to turn the last
determinant into the determinant of a $2\times2$ matrix kernel. Observe that, since $M_{a,t}$
and $P_{a+t^2}$ commute and $P_{a+t^2}\varrho_{a,t}=\varrho_{a,t}\bar P_{a+t^2}$,
the formula \eqref{eq:det1to2} can be rewritten as
\begin{multline}
\pp\!\left(\aip(x)-x^2\leq a\,\,\forall\,x\leq t,\,\aip(t+s)-(t+s)^2\leq b\right)\\
=\det\!\left(I-\K+\K(I-Q)e^{-sH}(I-P_2)e^{sH}\K\right)\label{eq:altDet}
\end{multline}
where
\begin{equation}
P_1=P_{a+t^2},\qquad P_2=P_{b+(t+s)^2},\qand
Q=P_1(I+M_{a,t}\varrho_{a,t})=P_{a+t^2}+M_{a,t}\varrho_{a,t}\bar P_{a+t^2}.\label{eq:projs}
\end{equation}
Note that $Q^2=Q$ (although $Q$ is not a projection in $L^2(\rr)$, as it is an unbounded
operator). This formula has exactly the same structure as the formula \eqref{eq:airyfd}
for the finite-dimensional distributions of the Airy$_2$ process (in the case $n=2$)
which, we recall, is equivalent to the extended kernel formula \eqref{eq:detform}. 
The equivalence of the two types of formulas was developed by \citet{prahoferSpohn} and
\citet{prolhacSpohn}, and later made rigorous and extended to the Airy$_1$ case by us in
\cite{quastelRemAiry1}. The same method will work for \eqref{eq:altDet}, and the
result is the following:

\begin{prop}\label{prop:matrixkernel}
  With the notation introduced above,
  \begin{multline}\label{eq:matrixdet}
    \pp\!\left(\aip(x)-x^2\leq a\,\,\forall\,x\leq t,\,\aip(t+s)-(t+s)^2\leq b\right)\\
    =\det\!\left(I-\Gamma\!\left[
      \begin{array}{cc}
        Q\K P_1 & Qe^{-sH}(\K-I)P_2\\
        P_2e^{sH}\K P_1 & P_2\K P_2
      \end{array}\right]\Gamma^{-1}\!\right)_{L^2(\rr)^2},
  \end{multline}
  where
  \[\Gamma=\left[\begin{array}{cc}
        G & 0\\
        0 & G
      \end{array}\right]\qquad\text{with}\qquad
    Gf(x)=e^{-2tx}\varphi^{-1}(x)f(x)\qand\varphi(x)=(1+x^2)^{1/2}.\]
    In particular, the operator above is identity plus trace class.
\end{prop}

The conjugation by $\Gamma$ is needed to make the operator trace class. Note that there is
a slight difference between this formula and the ones for the Airy$_1$ and Airy$_2$
processes: the formula is not written in the most symmetric way, as the first column in
the brackets is post-multiplied by $P_1$ instead of $Q$. Formally there is no difference,
because since $Q^2=Q$ one can pre-multiply the whole matrix by $\left[\begin{smallmatrix}
    Q & 0\\ 0 & Q\end{smallmatrix}\right]$ and then use the cyclic property of the
determinant to turn this into a post-factor of $P_1Q=Q$ for the first column. But this
form of the formula will turn out to be better for obtaining the desired bounds in Lemma
\ref{lem:kernelEst}.

The proof of Proposition \ref{prop:matrixkernel} follows the steps
of the proofs in \cite{prolhacSpohn,quastelRemAiry1}, but given the slight
difference noted above, and since it is short and easy to present in the two-dimensional
case, we include the details.

\begin{proof}[Proof of Proposition \ref{prop:matrixkernel}]
  We remark that all the manipulations performed on Fredholm determinants below rely on
  knowing that the operators inside each of them are identity plus trace class (see the
  proof of Theorem \ref{a1-thm:airy1} of \cite{quastelRemAiry1} for more details), but
  this can be seen in each case from Lemma \ref{lem:1to2} and Lemma \ref{lem:kernelEst}
  and similar estimates.

  Let $D$ denote the determinant in \eqref{eq:matrixdet}. The operator inside it can be
  factored as
  \begin{align}
    &\left[
      \begin{array}{cc}
        I & GQe^{-sH}P_2G^{-1}\\
        0 & I
      \end{array}\right]\\
    &\cdot\!\left(I-\Gamma\!\left[
        \begin{array}{cc}
          Q\K P_1-Qe^{-sH}P_2e^{sH}\K P_1 & Qe^{-sH}\K P_2-Qe^{-sH}P_2\K P_2\\
          P_2e^{sH}\K P_1 & P_2\K P_2
        \end{array}\right]\Gamma^{-1}\!\right).
  \end{align}
  Note that the determinant of the first matrix on the right side above is 1, so $D$ equals
  \begin{equation}
    \begin{multlined}
      \det\!\left(I-\Gamma\!\left[
          \begin{array}{cc}
            Q\K P_1-Qe^{-sH}P_2e^{sH}\K P_1 & Qe^{-sH}\K P_2-Qe^{-sH}P_2\K P_2\\
            P_2e^{sH}\K P_1 & P_2\K P_2
          \end{array}\right]\!\Gamma^{-1}\right)_{L^2(\rr)^2}\\
      =\det\!\left(I-\Gamma\!\left[
          \begin{array}{cc}
            Q\K-Qe^{-sH}P_2e^{sH}\K & 0\\
            P_2e^{sH}\K & 0
          \end{array}\right]
        \left[
          \begin{array}{cc}
            P_1 & e^{-sH}P_2\\
            0 & 0
          \end{array}\right]\!\Gamma^{-1}\right)_{L^2(\rr)^2}.
    \end{multlined}
  \end{equation}
  Since $\K$ is a projection we may pre-multiply each entry in the second bracket by $\K$
  and then use the cyclic property of the determinant to get
    \begin{align}
      D&=\det\!\left(I-\left[
      \begin{array}{cc}
        \K P_1 & \K e^{-sH}P_2\\
        0& 0
      \end{array}\right]\left[
        \begin{array}{cc}
          Q\K-Qe^{-sH}P_2e^{sH}\K & 0\\
          P_2e^{sH}\K & 0
        \end{array}\right]\right)_{L^2(\rr)^2}\\
      &=\det\!\left(I
      -\left[
      \begin{array}{cc}
        \K Q\K-\K Qe^{-sH}P_2e^{sH}\K+e^{-sH}\K P_2e^{sH}\K & 0\\
        0& 0
      \end{array}\right]\right)_{L^2(\rr)^2}\\
      &=\det\!\left(I-\K Q\K+\K Qe^{-sH}P_2e^{sH}\K-e^{-sH}\K P_2e^{sH}\K\right)_{L^2(\rr)}.
    \end{align}
    This last determinant equals the one on the right side of \eqref{eq:altDet}, and the
    result follows.
\end{proof}

The $2\times2$ matrix kernel formula is useful because it will allow us to extract easily
the first two factors in the bound given in Lemma \ref{lem:prtbd}. This idea was
introduced by \citet{widomAiry2}, where he studied the asymptotics in $t$ of
$\pp(\aip(0)\leq s_1,\,\aip(t)\leq s_2)$.

\begin{proof}[Proof of Lemma \ref{lem:prtbd}]
  We start with the formula in Proposition \ref{prop:matrixkernel}, with $a=b=\beta t^2$,
  and use the idea introduced in \cite{widomAiry2}: factor out the two diagonal terms in
  the determinant and then estimate the remainder. More precisely, we write
  \begin{equation}
  \begin{aligned}
   I&-\Gamma\!\left[
      \begin{array}{cc}
        Q\K P_1 & Qe^{-sH}(\K-I)P_2\\
        P_2e^{sH}\K P_1 & P_2\K P_2
      \end{array}\right]\!\Gamma^{-1}
    =\left(I-\Gamma\!\left[
      \begin{array}{cc}
        Q\K P_1 &0\\
        0 & P_2\K P_2
      \end{array}\right]\!\Gamma^{-1}\!\right)\\
  &\!\!\cdot\left(I-\Gamma\!\left[
      \begin{array}{cc}
       0 & (I-Q\K P_1)^{-1}Qe^{-sH}(\K-I)P_2\\
        (I-P_2\K P_2)^{-1}P_2e^{sH}\K P_1 & 0
      \end{array}\right]\!\Gamma^{-1}\!\right)
    \end{aligned}
  \end{equation}
  and then recognize that the determinant of the first factor on the right side equals
  \begin{multline}
    \det(I-GQ\K P_1G^{-1})\det(I-GP_2\K P_2G^{-1})\\
    =\pp\!\left(\aip(x)-x^2\leq0\,\,\forall\,x\leq
      t\right)\pp\!\left(\aip(t+s)-(t+s)^2\leq0\right).
  \end{multline}
  To get the last equality, observe first that $\det(I-GP_2\K P_2G^{-1})=F_{\rm GUE}((t+s)^2)$
  which is the second factor on the right side. For the first one we note that, by the
  cyclic property of determinants and the facts that $\K^2=\K$ and $P_1Q=Q$,
  \[\det(I-GQ\K P_1G^{-1})=\det(I-\K Q\K)=\pp\!\left(\aip(x)-x^2\leq0\,\,\forall\,x\leq
      t\right)\]
    by \oteqref{eq:limL}. So we are left with estimating
    \[\det\!\left(I-\Gamma\!\left[
      \begin{array}{cc}
       0 & (I-Q\K P_1)^{-1}Qe^{-sH}(\K-I)P_2\\
        (I-P_2\K P_2)^{-1}P_2e^{sH}\K P_1 & 0
      \end{array}\right]\!\Gamma^{-1}\right)_{L^2(\rr)^2}.\]

  The last determinant equals $\det(I-\widetilde K)$,
  with $\widetilde K=R_{1,1}R_{1,2}R_{2,2}R_{2,1}$ and
  \begin{equation}\label{eq:Rs}
    \begin{alignedat}{2}
      R_{1,1}&=G(I- Q\K P_1)^{-1}G^{-1},\qquad\qquad&&R_{1,2}=GQe^{-sH}(\K-I)P_2\\
      R_{2,2}&=(I-P_2\K P_2)^{-1},&& R_{2,1}=P_2e^{sH}\K P_1G^{-1}.
    \end{alignedat}
  \end{equation}
  Now $|\!\det(I-\widetilde K)-\det(I)|\leq\|\widetilde K\|_1e^{1+\|\widetilde
    K\|_1}\leq\|\widetilde K\|_1e^2$ for $\|\widetilde K\|_1\leq1$, so the proof will be
  complete once we show that
  \begin{equation}
    \label{eq:tildeK}
    \|\widetilde K\|_1\leq\tfrac1{2}e^{-2}a_2(t+s)^{-3}e^{-\frac43(\beta+1)^{3/2}(t+s)^3}.
  \end{equation}
  To get this estimate we use \eqref{eq:norms} to write
  \[\|\widetilde K\|_1\leq\|R_{1,1}\|_1\|R_{1,2}\|_1\|R_{2,2}\|_1\|R_{2,1}\|_1\]
  and use Lemma \ref{lem:kernelEst}, which gives, writing $\sigma=1+\beta$,
  \[\|\widetilde K\|_1\leq cs^{-3/2}t^{1/2}(t^2+s^2)^{3/2}e^{-\frac23(\sigma
    t^2+2ts+s^2)^{3/2}-\frac23\sigma^{3/2}t^3-s(\sigma t^2+2ts+s^2)}.\]
  Now taking $s=\alpha t$ we get
  \[\|\widetilde K\|_1\leq
  c\alpha^{-3/2}(1+\alpha^2)^{3/2}t^{2}e^{-\frac43(1+\alpha)^3\sigma^{3/2}t^3-h_\sigma(\alpha)t^3},\] where
  $h_\sigma(\alpha)=\frac23\sigma^{3/2}+\alpha(\sigma+2\alpha+\alpha^2)-\frac23(\sigma+2\alpha+\alpha^2)^{3/2}$.
  Observe that for fixed $\sigma=\beta+1>4$ we have $h_\sigma(0)=0$ and $h'_\sigma(0)>0$,
  which implies that $h_\sigma(\alpha)>0$ for small enough $\alpha$. Therefore
  \eqref{eq:tildeK} holds for small enough $\alpha$ and large enough $t$, and the result
  follows.
\end{proof}

\begin{lem}\label{lem:kernelEst}
  Let $R_{1,1}$, $R_{1,2}$, $R_{2,2}$ and $R_{2,1}$ be defined as in \eqref{eq:Rs}. Then
  there is a $c>0$ such that if $t$ and $s$ are large enough and $\sigma=\beta+1\geq4$,
  \begin{subequations}
    \begin{align}
      \label{eq:R11}&\|R_{1,1}\|_1\leq2,\\
      \label{eq:R12}&\|R_{1,2}\|_1\leq cs^{-1/2}t^{-1}(t^2+s^2)^2e^{-2\sigma  t^3-s(\sigma t^2+2ts+s^2)},\\
      \label{eq:R22}&\|R_{2,2}\|_1\leq2,\\
      \label{eq:R21}&\|R_{2,1}\|_1\leq cs^{-1}t^{3/2}(t^2+s^2)^{-1/2}e^{-\frac23(\sigma
        t^2+2ts+s^2)^{3/2}-\frac23\sigma^{3/2}t^3+2\sigma t^3}.
    \end{align}
  \end{subequations}
\end{lem}

The proof of this result is postponed to Section \ref{sec:kernelEst}.
  
\section{Estimates of operator norms}
\label{sec:extras}

We will use below the following well-known estimates for the Airy function (see
(10.4.59-60) in \cite{abrSteg}):
\begin{equation}
  |\!\Ai(x)|\leq C\,e^{-\frac23x^{3/2}}\quad\text{for $x>0$},\qquad
  |\!\Ai(x)|\leq C\,\quad\text{ for $x\leq0$}.\label{eq:airybd}
\end{equation}
We start some with some basic integral estimates involving the
Airy function:

\begin{lem}\label{lem:airyInt}
  There is a $c>0$ such that for any $m>0$ and $\alpha,\alpha',k,t\in\rr$ such that if $\alpha>0$ or
  $m\geq\frac14\alpha^2t^2$ we have, for large enough $t$,
  \begin{gather}
    \int_m^\infty dx\,x^ke^{-\alpha tx}\Ai(x)^2\leq ct^{2k-1}e^{-\alpha
      tm-\frac43m^{3/2}}\\
    \shortintertext{and}
    \int_m^\infty dx\int_0^\infty dy\,x^ke^{-\alpha tx}\Ai(x+y)^2e^{\alpha' y}\leq ct^{2k-1}e^{-\alpha
      tm-\frac43m^{3/2}}
  \end{gather}
\end{lem}

\begin{proof}
  Using \eqref{eq:airybd} the first integral is bounded by
  \[c\int_m^\infty dx\,x^{k}e^{-\alpha tx-\frac43x^{3/2}}=ct^{2(k+1)}\int_{mt^{-2}}^\infty
  dx\,x^{k}e^{-(\alpha x+\frac43x^{3/2})t^3}.\] The exponent is maximized for $x\geq0$ at
  $x=\frac14\alpha^2\leq m$ if $\alpha<0$ and at 0 otherwise, so the first estimate
  follows from a simple application of Laplace's method, see Lemma \gref{lem:laplace}. The
  second integral is bounded in the same way after noting that $\Ai(x+y)\leq
  c\,e^{-\frac23x^{3/2}-\frac23y^{3/2}}$ for $x,y\geq0$.
\end{proof}

\subsection{Estimates used for the upper bound}

\begin{lem}\label{lem:Qs}
    Let $Q_1$ and $Q_2$ be defined as in \eqref{eq:A11A12}. Then there is a $c>0$ such
    that for $m\geq-2t$ and $t\geq1$ we have
    \[\|Q_1\|_2\leq ct^{3/4}\qqand\|Q_2\|_2\leq ct^{3/4}.\]
\end{lem}

\begin{proof}
  Writing $\tilde x=2^{1/3}x$ we have
  \begin{align}
    \|Q_1\|^2_2&=\int_0^\infty dx\int_{-\infty}^\infty d\lambda\,e^{\ep\lambda}\Ai(\tilde
    x+m+\lambda)^2=\int_0^\infty dx\,e^{-\ep\tilde x-\ep m}\int_{-\infty}^\infty
    d\lambda\,e^{\ep\lambda}\Ai(\lambda)^2\\
    &\leq ct\int_{-\infty}^\infty
    d\lambda\,e^{\lambda/t}\Ai(\lambda)^2
  \end{align}
  by our assumption $m\geq-2t$ and the facts that $\ep=t^{-1}$ and $t\geq1$. Using the
  estimate $|\!\Ai(\lambda)|\leq c|\lambda|^{-1/4}$ as $\lambda\to-\infty$ (see (10.4.60)
  in \cite{abrSteg}) and \eqref{eq:airybd} we deduce that $\|Q_1\|^2_2\leq ct^{3/2}$. The
  bound for $Q_2$ is proved in exactly the same way.
\end{proof}

\begin{lem}\label{lem:psi2bd}
  For large enough $t>0$ and $s>t$ we have
  \begin{gather}
    \int_{-2t}^\infty dm\,\|P_0\psi_{s,m}(2^{1/3}\cdot)\|^2_2\leq e^{-\frac43s^3+4st}\\
    \shortintertext{and}
    \int_{-2|t|}^\infty dm\,\|P_0\psi_{-s,m}(2^{1/3}\cdot)\|^2_2\leq e^{-\frac43s^3}\\
  \end{gather}

\end{lem}

\begin{proof}
  Since $|\!\Ai'(x)|$ satisfies the same bound \eqref{eq:airybd} as $\Ai(x)$ for $x>0$ (see
  (10.4.61) in \cite{abrSteg}) we have for $s\geq t>2$ and $m\geq-2t$ that
  \begin{align}
    \|P_0\psi_{s,m}(2^{1/3}\cdot)\|^2_2&=\int_0^\infty
    dx\,4e^{2xs}\left[s\Ai(x+m+s^2)+\Ai'(x+m+s^2)\right]^2\\
    &\leq c(1+s^2)\int_0^\infty dx\,e^{2xs-\frac43(x+m+s^2)^{3/2}}.
  \end{align}
  Integrating over $m\geq-2t$ and scaling $m$ and $x$ by $s^2$ we get
  \begin{align}
    \int_{-2t}^\infty dm\,\|P_0\psi_{s,m}(2^{1/3}\cdot)\|^2_2&\leq
    cs^6\int_{-2ts^{-2}}^\infty dm\int_{0}^\infty dx \,e^{2xs^3-\frac43(x+m+1)^{3/2}s^3}\\
    &=cs^6\int_0^\infty dm\int_{0}^\infty dx
    \,e^{2xs^3-\frac43(x+m+1-2ts^{-2})^{3/2}s^3}\\
    &\leq cs^6\int_0^\infty dm\int_{0}^\infty dx
    \,e^{[2x-\frac43(x+m+1)^{3/2}]s^3}e^{4ts\sqrt{x+m+1}},
  \end{align}
  where in the last line we used the inequality
  $\frac43(x+m+1-2ts^{-2})^{3/2}\geq\frac43(x+m+1)^{3/2}-4ts^{-2}\sqrt{x+m+1}$ for $x>0$,
  $m\geq-2t$ and $s\geq t>2$. The term in brackets in the first exponential in the last
  integral is maximized at $x=m=0$,
  and then applying Laplace's method as in the proof of Lemma \ref{lem:airyInt} gives
  \[\int_{0}^\infty dm\,\|P_0\psi_{s,m}(2^{1/3}\cdot)\|^2_2\leq
  cs^{3}e^{-\frac43s^3+4ts}.\] This gives the first bound. The second bound is similar
  (and slightly simpler).
\end{proof}

\begin{lem}\label{lem:tracenormbd}
  Let $\Omega^m_L$ be the operator defined in \geqref{eq:OmegaL} for $m,L>0$ (here we are
  making the dependence on $m$ explicit in the notation). Then there is an $\eta>\frac43$
  satisfying the following: for fixed $L>0$ there is a $c>0$ such that for all $m>0$
  \[\left\|e^{LH}\K\Omega^m_Le^{LH}\K\right\|_1\leq c\,e^{-\eta m^{3/2}}.\]
\end{lem} 
  
\begin{proof}
  The proof of this result can be adapted from the proof of Lemmas \grefn{lem:tildeRL2to0}
  and \gref{lem:enie0}. In that result it is only proved that the above norm is finite,
  but one can get the above estimate by carefully keeping track of the dependence in
  $m$. We leave the details to the reader.
\end{proof}

\subsection{Proof of Lemma \ref{lem:kernelEst}}\label{sec:kernelEst}

We will use Lemma \ref{lem:airyInt} repeatedly without reference. The assumption
$\sigma=\beta+1\geq4$ enters crucially as it ensures in each case that the hypothesis of
the lemma holds. Throughout the proof we will assume that $t\geq t_0$, where $t_0>1$
should be taken as large as needed to make the estimates work.

Recall the notation introduced in \eqref{eq:projs}. In the present case we have $a=b=\beta
t^2$, and thus recalling that $\sigma=1+\beta$ and writing $r^2=\sigma t^2+2ts+s^2$,
$M=M_{\beta t^2,t}$ and $\varrho=\varrho_{\beta t^2,t}$ to simplify the notation, we have
\[P_1=P_{\sigma t^2},\qquad P_2=P_{r^2}\qqand Q=P_1(I+M\varrho).\]
We also define the multiplication operator
\[Nf(x)=\varphi(x)^{-1}f(x),\] where, as before, $\varphi(x)=(1+x^2)^{1/2}$. Finally, we
will use repeatedly the decomposition
\[\K=B_0P_0B_0,\]
where, we recall, $B_0(x,y)=\Ai(x+y)$.

Let us start with the first estimate. Since 
$Q=P_1+P_1M\varrho $, we have
\begin{equation}
  \|GQKP_1G^{-1}\|_1\leq\|GP_1B_0P_0\|_2\|P_0B_0P_1G^{-1}\|_2
  +\|GP_1M\varrho B_0P_0N\|_2\|N^{-1}P_0B_0P_1G^{-1}\|_2.\label{eq:firstR11}
\end{equation}
Now
\begin{equation}
  \|GP_1B_0P_0\|^2_2=\int_{\sigma  t^2}^\infty dx\int_0^\infty
  dy\,\varphi(x)^{-2}e^{-4tx}\Ai(x+y)^2
  \leq ct^{-5}e^{-4\sigma  t^3-\frac{4}3\sigma^{3/2}t^3},\label{eq:GP1B0P0}
\end{equation}
while, recalling that $M\varrho f(x)=e^{2t(x-\sigma t^2)}f(2\sigma t^2-x)$, we have
\begin{align}
  \|GP_1M\varrho B_0P_0N\|^2_2 &=\int_{\sigma t^2}^\infty dx\int_0^\infty
  dy\,\varphi(x)^{-2}e^{-4\sigma  t^3}\Ai(2\sigma  t^2-x+y)^2\varphi(y)^{-2}\\
  &\leq e^{-4\sigma t^3}\|\!\Ai\|^2_\infty\|P_1\varphi^{-1}\|^2_2\|\varphi^{-1}\|^2_2\leq
  ct^{-2}e^{-4\sigma t^3}.
\end{align}
Similarly
\[\|P_0B_0P_1G^{-1}\|^2_2=\int_{0}^\infty dx\int_{\sigma t^2}^\infty
dy\,\varphi(y)^{2}e^{4ty}\Ai(x+y)^2 \leq ct^3e^{4\sigma t^3-\frac{4}3\sigma^{3/2}t^3},\]
and one can easily see that the same estimate holds with a possibly larger constant for
$\|N^{-1}P_0B_0P_1G^{-1}\|^2_2$. Putting these estimates together with \eqref{eq:firstR11}
we deduce that
\[\|GQKP_1G^{-1}\|_1\leq cte^{-\frac43\sigma^{3/2}t^3}<\tfrac12\] for large enough $t$, and
then
\begin{equation}
  \|R_{1,1}\|_1\leq\sum_{k\geq0}\|(GQ\K P_1G^{-1})^{k}\|_1
  \leq\sum_{k\geq0}\|GQ\K P_1G^{-1}\|_1^k<2,
\end{equation}
which gives \eqref{eq:R11}.

We turn now to $R_{1,2}$. Since $e^{-sH}(\K-I)$ has integral kernel (in $x,y$) given by
$\int_{-\infty}^0d\lambda\,e^{s\lambda}\Ai(x+\lambda)\Ai(y+\lambda)$ (see \eqref{eq:eLHK}
and the paragraph around it), we may use the decomposition $Qe^{-sH}(\K-I)P_2=QB_0\bar
P_0e^{s\xi}B_0P_2$, where $e^{a\xi}$ is the multiplication operator defined by
$(e^{a\xi}f)(x)=e^{ax}f(x)$, so that
\begin{equation}
  \|R_{1,2}\|_1\leq\|GQB_0\bar P_0 N\|_2\|N^{-1}\bar P_0e^{s\xi}B_0P_2\|_2.\label{eq:firstR12}
\end{equation}
Now
\begin{equation}
  \begin{aligned}
    \|GP_1B_0\bar P_0N\|_2^2 &=\int_{\sigma t^2}^\infty dx\int_{-\infty}^0dy\,
    e^{-4tx}\varphi(x)^{-2}\Ai(x+y)^2\varphi(y)^{-2}\\
    &\leq\|\!\Ai\|^2_\infty\|P_1\varphi^{-1}\|_2^2\|\varphi^{-1}\|_2^2\,e^{-4\sigma t^3} \leq
    ct^{-2}e^{-4\sigma t^3},
  \end{aligned}
\end{equation}
while
\begin{equation}
  \begin{aligned}
    \|GP_1M\varrho B_0\bar P_0N\|_2^2 =\int_{\sigma t^2}^\infty
    dx\int_{-\infty}^0dy\, e^{-4\sigma t^3}\varphi(x)^{-2}\Ai(2\sigma
    t^2-x+y)^2\varphi(y)^{-2} \leq ct^{-2}e^{-4\sigma t^3}
  \end{aligned}
\end{equation}
in a similar way. On the other hand
\begin{align}
  \|N^{-1}\bar P_0e^{s\xi}B_0P_2\|_2^2&=\int_{-\infty}^0dx\int_{r^2}^\infty
  dy\,\varphi(x)^{2}e^{2sx}\Ai(x+y)^2\\
  &=\int_{r^2}^\infty dy\,e^{-2sy}\int_{-\infty}^y
  dx\,(1+(x-y)^{2})e^{2sx}\Ai(x)^2.
\end{align}
We split the $x$ integral into the regions $(-\infty,0]$ and $(0,y]$. On the first one we
can estimate the integral by
\[\|\!\Ai\|^2_\infty\int_{r^2}^\infty dy\,e^{-2sy}\int_{-\infty}^0
  dx\,(1+(x-y)^{2})e^{2sx}\leq cr^4s^{-2}e^{-2sr^2},\]
while on the second one we estimate by
\[c\int_{r^2}^\infty dy\,e^{-2sy}\int_0^y dx\,(1+(x-y)^{2})e^{-\frac43x^{3/2}+2sx}
  \leq cr^4s^{-1}e^{-2sr^2},\]
giving \[\|N^{-1}\bar P_0e^{s\xi}B_0P_2\|_2^2\leq cr^4s^{-1}e^{-2sr^2}.\]
Putting the three bounds in \eqref{eq:firstR12} gives \eqref{eq:R12}.

For $R_{2,2}$ we observe that $\|P_2\K P_2\|_1\leq\|P_2B_0P_0\|_2\|P_0B_0P_2\|_2$, which
can easily be seen to be bounded by $\frac12$ for large enough $t$ by bounds similar to
(and simpler than) those used to prove \eqref{eq:R11}, and thus we get \eqref{eq:R22} in
exactly the same way.

Finally, for $R_{2,1}$ we use a similar decomposition as for $R_{1,2}$: using
\eqref{eq:eLHK} we may write
\[\|P_2e^{sH}\K P_1G^{-1}\|_1\leq\|P_2B_0e^{-s\xi/2}P_0\|_2\|P_0e^{-s\xi/2}B_0P_1G^{-1}\|_2.\]
Now
\[\|P_2B_0e^{-s\xi/2}P_0\|^2_2 =\int_{r^2}^\infty dx\int_0^\infty dy\,
e^{-sy}\Ai(x+y)^2\leq cs^{-1}r^{-1}e^{-\frac43r^3}\] and
\[\|P_0e^{-s\xi/2}B_0P_1G^{-1}\|_2^2= \int_{0}^\infty dx\int_{\sigma t^2}^\infty dy\,
e^{-sx}\Ai(x+y)^2\varphi(y)^2e^{4ty} \leq cs^{-1}t^{3}e^{4\sigma
  t^3-\frac43\sigma^{3/2}t^3},\] which gives \eqref{eq:R21}.

\printbibliography[heading=apa]

\end{document}

